\documentclass[11pt, reqno]{amsart}

\usepackage[utf8]{inputenc}
\usepackage[english]{babel}
\usepackage{amsaddr}
\usepackage{lmodern}
\usepackage{enumitem}

\usepackage[margin=1in]{geometry}
\usepackage{amsmath, amsthm, amsfonts, amssymb, tikz-cd, mathtools, mathrsfs,theoremref, txfonts}

\usepackage[T1]{fontenc}
\usepackage{graphicx}
\usepackage{hyperref}

\newtheorem{theorem}{Theorem}[section]

\newtheorem{corollary}[theorem]{Corollary}
\newtheorem{lemma}[theorem]{Lemma}

\newtheorem{obs}[theorem]{Observation}
\theoremstyle{definition}

\usepackage{thm-restate}

\newcommand{\F}{\mathbb{F}}

\newcommand{\Z}{\mathbb{Z}}

\DeclareMathOperator{\rank}{rank}

\DeclareMathOperator{\supp}{supp}
\DeclareMathOperator{\im}{im}

\newcommand{\abs}[1]{\left\lvert#1\right\rvert}

\title{S\'ark\"ozy's Theorem in Various Finite Field Settings}
\author{Anqi Li and Lisa Sauermann}\thanks{Li was supported by Mariana Polonsky Slocum (1955) Memorial Fund as part of the MIT Undergraduate Research Opportunities Program (UROP). Sauermann was supported by NSF Award DMS-2100157 and a Sloan Research Fellowship.} 
\address{Massachusetts Institute of Technology, Cambridge, MA 02139, USA}
\email{\{anqili,lsauerma\}@mit.edu}

\begin{document}

\maketitle
\begin{abstract}
    In this paper, we strengthen a result by Green about an analogue of S\'ark\"ozy's theorem in the setting of polynomial rings $\mathbb{F}_q[x]$. In the integer setting, for a given polynomial $F \in \mathbb{Z}[x]$ with constant term zero, (a generalization of) S\'ark\"ozy's theorem gives an upper bound on the maximum size of a subset $A \subset \{1, \ldots, n \}$ that does not contain distinct $a_1,a_2 \in A$ satisfying $a_1 - a_2 = F(b)$ for some $ b \in \mathbb{Z}$. Green proved an analogous result with much stronger bounds in the setting of subsets $A \subset \mathbb{F}_q[x]$ of the polynomial ring $\mathbb{F}_q[x]$, but required the additional condition that the number of roots of the polynomial $F \in \mathbb{F}_q[x]$ is coprime to $q$. We generalize Green’s result, removing this condition. As an application, we also obtain a version of S\'ark\"ozy’s theorem with similarly strong bounds for subsets $A \subset \mathbb{F}_q$ for $q = p^n$ for a fixed prime $p$ and large $n$. 
\end{abstract}

\section{Introduction}

In this paper, we study variants of S\'ark\"ozy's theorem \cite{S78}.

\begin{theorem}[S\'ark\"ozy's Theorem]\label{thm:Sarkozy}
Let $\alpha(n)$ be the maximum size of a subset $A \subset \{1, 2, \ldots, n\}$ such that there do not exist distinct $a_1, a_2 \in A$ with $a_1 - a_2 = b^2$ for some $b \in \mathbb{Z}$. Then $\lim \limits_{n \rightarrow \infty} \alpha(n)/n = 0$.
\end{theorem}

A natural generalization of this theorem is to replace $b^2$ by another polynomial $F(b)$, yielding the following result (observed for example in \cite{KM78}). 

\begin{theorem}[Generalization of S\'ark\"ozy's Theorem]\label{thm:Sarkozy2}
Let $F \in \Z[x]$ be a polynomial of degree $k$ with constant term zero. Let $\beta_{k}(n)$ be the maximum size of a subset $A \subset \{1, \ldots, n \}$ such that there do not exist distinct $a_1, a_2 \in A$ with $a_1 - a_2 = F(b)$ for some $b \in \Z$. Then $\lim \limits_{n \rightarrow \infty} \beta_k(n)/n = 0$.
\end{theorem}

The best known quantitative bounds for $\alpha(n)$ and $\beta(n)$ improve upon the trivial bounds $\alpha(n) \leq n$ and $\beta_k(n) \leq n$ by polylogarithmic factors. Specifically, the best known bound for Theorem~\ref{thm:Sarkozy} is $\alpha(n) \le O(n/(\log n)^{c \log \log \log n})$ for some absolute constant $c>0$ due to Bloom and Maynard \cite{BM22}. Building upon work of Pintz--Steiger--Szemer\'edi \cite{PSS88}, Rice \cite{R19} obtained the bound $\beta_k(n) \le O(n/(\log n)^{c(k) \log \log \log \log n})$ for Theorem~\ref{thm:Sarkozy2} for some constant $c(k)>0$ depending on $k$.

A few years ago, Green \cite{G17}\footnote{As the assumptions in \cite[Theorem 1.2]{G17} are not stated correctly in the published version, we cite the newer and corrected arXiv version of the paper.} considered the following analogue of Theorem~\ref{thm:Sarkozy2} for polynomial rings (an analogue of Theorem~\ref{thm:Sarkozy} in this setting with much weaker quantitative bounds was shown earlier by L\^{e} and Liu \cite{LL13}).

\begin{theorem}[{\cite[Theorem 1.2]{G17}}]\label{thm:Green-S}
Let $q$ be a prime power and let $F \in \F_q[x]$ be a polynomial of degree $k$ with constant term zero. Suppose the number of roots of $F$ in $\F_q$ is coprime to $q$. Let $\gamma_{q,k}(n)$ be the maximum size of a set $A$ of polynomials with degree less than $n$ in $\F_q[x]$ such that there do not exist distinct polynomials $p_1(x), p_2(x) \in A$ with $p_1(x) - p_2(x) = F(b(x))$ for some $b(x) \in \F_q[x]$. Then there exists a constant $t_{q,k}< q$ such that $\gamma_{q,k}(n)\le 2 \cdot (t_{q,k})^n$.
\end{theorem}

We observe that one trivially has $\gamma_{q,k}(n) \leq q^n$, since the total number of polynomials of degree less than $n$ in $\F_q[x]$ is $q^n$. In light of this, the shape of the Green's bound is drastically different from those mentioned earlier in the integer setting. Recall that in the integer setting the best known bounds are polylogarithmic saving over the trivial bound $n$. In contrast, in the polynomial ring setting, Theorem~\ref{thm:Green-S} gives a much better power saving bound over the trivial bound $q^n$. To obtain these strong bounds in the polynomial setting, Green utilized the Croot--Lev--Pach \cite{CLP17} polynomial method. Croot--Lev--Pach \cite{CLP17} introduced this new polynomial technique to obtain a new power saving upper bound for the size of subsets in $\Z_4^n$ without three-term arithmetic progressions. This polynomial method has found many applications, leading to several important breakthroughs such as the groundbreaking power saving upper bound on the capset problem by Ellenberg--Gijswijt \cite{EG17}. 

In Theorem~\ref{thm:Green-S}, it is natural to only consider polynomials with constant term zero. Indeed, suppose $F(T) = T^q - T +1$ and let $A$ be the set of polynomials with degree less than $n$ in $\F_q[x]$ with constant term zero. Then $|A| = q^{n-1}$ and there do not exist $p_1(x), p_2(x) \in A$ and $b(x) \in \F_q[x]$ with $p_1(x) - p_2(x) = F(b(x))$. This is because $F(b(x))$ always has constant term 1 for any $b(x) \in \F_q[x]$ while $p_1(x) - p_2(x)$ has constant term 0 for all $p_1(x), p_2(x) \in A$. Now, $|A|=q^{n-1}$ is a constant fraction of the total number of polynomials in $\F_q[x]$ with degree less than $n$. So we cannot hope for a good bound on $|A|$ in Theorem~\ref{thm:Green-S} without the assumption that $F$ has constant term zero. Similarly in Theorem~\ref{thm:Sarkozy2}, the constant term zero condition cannot be removed.

However, the condition in Theorem~\ref{thm:Green-S} on the number of roots of $F$ being coprime to $q$ is not as natural, and is an artefact of Green's proof. In this paper, we strengthen Theorem~\ref{thm:Green-S} by showing that the condition on the number of roots of $F$ is unnecessary.


\begin{theorem}\label{thm:S}
Let $q$ be a prime power and let $F \in \F_q[x]$ be a polynomial of degree $k$ with constant term zero.  Let $\gamma_{q,k}(n)$ be the maximum size of a set $A$ of polynomials with degree less than $n$ in $\F_q[x]$ such that there do not exist distinct polynomials $p_1(x),p_2(x) \in A$ with $p_1(x) - p_2(x) = F(b(x))$ for some $b(x) \in \F_q[x]$. Then there exist constants $0<t_{q,k} < q$ and $c_{q,k} > 0$ such that $\gamma_{q,k}(n) \le c_{q,k} \cdot (t_{q,k})^n$ holds for all $n$.
\end{theorem}

Our proof gives the following value for $t_{q,k}$:
\begin{equation}\label{eq:t-exp}
 t_{q,k} = \inf\limits_{0 < x < 1} \frac{1+x+ \cdots + x^{q-1}}{x^{\frac{1}{2}(q-1)(1- 1/(kd))}},
\end{equation}
where $d=\min \{k, (q-1)(1 + \log_q k)\}$. It is not hard to show that this value $t_{q,k}$ satisfies $t_{q,k} < q$. Indeed, when $x = 1$, the expression on the right hand side evaluates to $q$. Furthermore, the derivative of the expression is positive at $x=1$. Consequently, the infimum of this expression over $0 < x < 1$ is strictly less than $q$, meaning that $t_{q,k} < q$. We remark that optimizing the bounds in Green's proof in \cite{G17} gives the same value of $t_{q,k}$ in Theorem~\ref{thm:Green-S} as in (\ref{eq:t-exp}).

An explicit example of a polynomial $F$ to which Theorem~\ref{thm:S} but not Theorem~\ref{thm:Green-S} applies is the following: Let $q = p^n$ for a prime $p$ and consider the polynomial $F(x) = x^{k-p+1}(x-1)\ldots (x-(p-1))$ for any $k\ge p$. Then $F$ has exactly $p$ roots in $\mathbb{F}_q$ and so it does not satisfy the  condition in Theorem~\ref{thm:Green-S} on the number of roots of $F$ being coprime to $q$.

Another setting similar to $\F_q[x]$ in which one may consider a S\'ark\"ozy-style problem is $\F_{q}$, for a prime power $q = p^n$. We obtain the following result in the setting of $\F_q$, as an application of Theorem~\ref{thm:S}. 

\begin{corollary}\label{cor:Fq}
Let $p$ be a prime, $q = p^n$ be a prime power and $F \in \F_{p}[x]$ be a polynomial of degree $k$ with constant term zero. Let $\eta_{p,k}(n)$ be the size of the maximum subset $A \subset \F_q$ that does not contain distinct $a_1,a_2 \in A$ such that $a_1 - a_2 = F(b)$ for some $b \in \F_q$. Then there exist constants $0<t_{p,k} < p$ and $c_{p,k} > 0$ such that $ \eta_{p,k}(n) \le c_{p,k}\cdot (t_{p,k})^n$ holds for all $n$.
\end{corollary}

Again, the value of $t_{p,k}$ is as given by (\ref{eq:t-exp}). In this direction, Peluse \cite{P18, P19} studied polynomial patterns in $\F_q$ for $q$ of sufficiently large characteristic depending on the polynomial pattern. Using Fourier analytic techniques, Peluse \cite{P18} proved a power-saving bound on sets $A \subset \F_q$ not containing polynomial progressions $a, a+P_1(b), \ldots, a + P_r(b)$ for some $a,b \in \F_q$ with $b\neq 0$, for given linearly independent polynomials $P_1(x), \ldots, P_r(x)$ in $\Z[x]$ with constant term zero, assuming that $q$ has large characteristic. The setting in Corollary~\ref{cor:Fq} corresponds to the $r=1$ setting of Peluse's result. By applying the Croot--Lev--Pach polynomial method, we prove a power-saving bound in Corollary~\ref{cor:Fq} in a complementary regime to Peluse's result: while Peluse's theorem holds for $\F_q$ (where $q = p^n$) with sufficiently large characteristic $p$, Corollary~\ref{cor:Fq} holds for $\F_q$ (where $q = p^n$) with any fixed characteristic $p$ but with $n$ large.

\medskip
\textit{Organization}. In Section 2, we collection some preliminary tools in order to apply the Croot--Lev--Pach polynomial method. We prove Theorem~\ref{thm:S} in Section 3 and Corollary~\ref{cor:Fq} in Section 4. 

\section{Preliminary tools}

In this section, we collect some useful results to be applied later. 

\begin{lemma}\label{lem:weight1}
Let $S \subset \F_q^m$ with $|S| \geq 1$. Then there exists a polynomial $\mu\in \F_q[x_1, \ldots, x_m]$ of degree at most $ \abs{S} -1$ such that $\sum_{a \in S} \mu(a) \neq 0$.
\end{lemma}

\begin{proof}
We explicitly construct such a polynomial $\mu$. Suppose that the elements of $S$ are $v^{(1)}, \ldots, v^{(\abs{S})} \in \F_q^m$. For $j = 2, \ldots, |S|$, let $\ell_j \in \F_q[x_1, \ldots, x_m]$ be a linear polynomial corresponding to the equation of a hyperplane passing through $v^{(j)}$ but not passing through $v^{(1)}$. Consider $\mu(x_1, \ldots, x_m) := \prod_{j=2}^{\abs{S}} \ell_j$. By construction, $\deg \mu = \abs{S} - 1$, $\mu(v^{(1)}) \neq 0$ and $\mu( v^{(j)}) = 0$ for all $2 \leq j \leq \abs{S}$. Hence, it follows that $\sum_{a \in S}\mu(a) = \sum_{j=1}^{|S|} \mu(v^{(j)}) \neq 0$.
\end{proof}

In the following lemma, given a polynomial $\mu$, we construct a suitable polynomial to facilitate applying the Croot--Lev--Pach polynomial method \cite{CLP17} in our setting. For a polynomial $P(x) \in \F_q[x_1, \ldots, x_n]$, we define its support to be $\supp(P) = \{ x \in \F_q^n: P(x) \neq 0 \}$. 

\begin{lemma}\label{lem:Poly}
Let $q$ be a prime power and $d$ be a positive integer. Consider a polynomial map $\Phi:= (\phi_1, \ldots, \phi_n): \F_q^m \to \F_q^n$ where $\phi_i \in \F_q[x_1, \ldots x_m]$ and $\deg \phi_i \leq d$ for all $1 \leq i \leq n$. Assume further that there exists some $\mu \in \F_q[x_1, \ldots, x_m]$ satisfying $\sum_{a \in \Phi^{-1}(0)} \mu(a) \neq 0$. Then there exists a polynomial $P \in \F_q[x_1, \ldots, x_n]$ of degree at most $(q-1)(n - m/d) + (\deg \mu)/d$ such that $P(0) \neq 0$ and $\supp(P) \subset \im(\Phi)$.
\end{lemma}

In the proof of this lemma, we construct such a polynomial $P$ explicitly. To bound its degree, we utilize the following observation on vanishing power sums. 

\begin{obs}\label{obs:sum}
Let $q$ be a prime power and $0 \leq k < q-1$ be an integer. Then $\sum_{x \in \F_q} x^k = 0$.
\end{obs}

Here, we use the usual convention that $0^0=1$.

\begin{proof}
For $k=0$, note that $\sum_{x \in \F_q} x^k = \sum_{x \in \F_q} 1=0$. For $1\le k<q-1$, recall that $\F_q^{\times}$ is cyclic, and let $\xi$ be a generator of $\F_q^\times$. Then we have the geometric series
\[ \sum_{x \in \F_q} x^k = \sum_{i=0}^{q-2} \xi^{ki} = \frac{1 - \xi^{k(q-1)}}{1 - \xi^k} = 0.\qedhere\]
\end{proof}

\begin{proof}[Proof of Lemma~\ref{lem:Poly}]
Construct the polynomial $P \in \F_q[x_1, \ldots, x_n]$ as follows:

\begin{equation}\label{eq:poly}
    P(x_1, \ldots, x_n) = \sum_{a \in \F_q^m} \mu(a) \prod_{i=1}^{n} \left( 1 - (x_i - \phi_i(a))^{q-1} \right).
\end{equation}

We claim that $P(b) = \sum_{a \in \Phi^{-1}(b)} \mu(a)$ for every $b \in \F_q^n$. This is because if $a \in \Phi^{-1}(b)$, then $b_i = \phi_i(a)$ for all $ 1 \leq i \leq n$ and so $\prod_{i=1}^{n}( 1 - (b_i - \phi_i(a))^{q-1} ) = 1$. Conversely, if $a \not \in \Phi^{-1}(b)$ then there is some index $j$ such that $b_j \neq \phi_j(a)$ and so $1 - (b_j - \phi_j(a))^{q-1} = 0$. In particular, $\prod_{i=1}^{n}( 1 - (b_i - \phi_i(a))^{q-1} ) = 0$ since the $j$th term vanishes. Consequently, it follows that $P(b) = \sum_{a \in \Phi^{-1}(b)} \mu(a)$ for all $a \in \F_q^n$.

It follows by the condition for $\mu$ that we have $P(0) = \sum_{a \in \Phi^{-1}(0)} \mu(a) \neq 0$. Furthermore, since $\Phi^{-1}(b) = \emptyset$ for $b \not \in \mathrm{Im}(\Phi)$, we have $P(b) = 0$ for all $b \not \in \mathrm{Im}(\Phi)$. 

It remains to check that $\deg P\le (q-1)(n - m/d) + (\deg \mu)/d$.

Let us consider each of the terms
\[Q(a_1, \ldots, a_m,x_1, \ldots, x_n) := \mu(a_1, \ldots, a_m)\prod_{i=1}^{n} (1 - (x_i - \phi_i(a_1, \ldots, a_m))^{q-1}) \]
in (\ref{eq:poly}) as a polynomial in $\F_q[a_1, \ldots, a_m, x_1, \ldots, x_n]$. Note that $P(x_1, \ldots, x_n) = \sum_{a \in \F_q^m}  Q(a_1, \ldots, a_m,x_1, \ldots, x_n)$. Furthermore, we introduce a nonstandard weighting of the polynomial ring $\F_q[a_1, \ldots, a_m, x_1, \ldots, x_n]$, which we denote by $\deg^*$. While we continue to view each $x_i$ as having degree $\deg^*(x_i)=1$, we view each $a_i$ as having degree $\deg^*(a_i)=1/d$. Then we for example have $\deg^*(x_i^{r} a_i^{s}) = r+ s/d$. Note that, under this new weighting, we have $\deg^*(x_i - \phi_i(a_1,\ldots,a_m)) \leq 1$, since $\deg \phi_i \leq d$. Consequently, it follows that 
\begin{equation}\label{eq:degQ}
    \deg^* Q(a_1, \ldots, a_m, x_1, \ldots, x_n) \leq (\deg \mu)/d + (q-1)n.
\end{equation}

Each monomial in $Q(a_1, \ldots, a_m, x_1, \ldots, x_n)$ is of the form $a_1^{i_1} \cdots a_m^{i_m}x_1^{j_1} \cdots x_n^{j_n}$ with non-negative integers $i_1, \ldots, i_m, j_1,\ldots,j_m$. We claim that in the sum $P(x_1, \ldots, x_n) = \sum_{a \in \F_q^m}  Q(a_1, \ldots, a_m,x_1, \ldots, x_n)$, the contributions from the monomials $a_1^{i_1} \cdots a_m^{i_m}x_1^{j_1} \cdots x_n^{j_n}$ with $i_1 + \cdots + i_m < (q-1)m$ vanish. Indeed, if $i_1 + \cdots + i_m < (q-1)m$, there exists some index $s$ with $0\le i_s < q-1$. By Observation~\ref{obs:sum} we have $\sum_{a_{s} \in \F_q} a_{s}^{i_s} = 0$, and therefore
\[ \sum_{a \in \F_q^m} a_1^{i_1} \ldots a_m^{i_m}x_1^{j_1} \cdots x_n^{j_n}  = \left(\sum_{a_{s} \in \F_q} a_{s}^{i_s} \right) \left( \sum_{(a_1, \ldots, a_{s-1}, a_{s+1}, \ldots, a_m)\in \F_q^{m-1}} a_1^{i_1} \ldots a_{s-1}^{i_{s-1}} a_{s+1}^{i_{s+1}} \ldots a_m^{i_m} x_1^{j_1} \cdots x_n^{j_n}\right) = 0. \]
Thus, all monomials $a_1^{i_1} \cdots a_m^{i_m}x_1^{j_1} \cdots x_n^{j_n}$ with $i_1 + \cdots + i_m < (q-1)m$ cancel when we take the summation over all $a\in \F_q^m$. This means that all monomials in $P(x_1, \ldots, x_n) = \sum_{a \in \F_q^m}  Q(a_1, \ldots, a_m,x_1, \ldots, x_n)$ are obtained from monomials $a_1^{i_1} \cdots a_m^{i_m}x_1^{j_1} \cdots x_n^{j_n}$ in $Q(a_1, \ldots, a_m, x_1, \ldots, x_n)$ such that $i_1+\cdots + i_m \geq (q-1)m$. By (\ref{eq:degQ}), for these monomials we have
\[j_1 + \cdots +j_n \leq \deg^*(a_1^{i_1} \cdots a_m^{i_m}x_1^{j_1} \cdots x_n^{j_n})- (q-1)m/d\le  (\deg \mu)/d + (q-1)n - (q-1)m/d.\]
Thus, all monomials with non-zero coefficients in $P(x_1,\dots,x_n)$ are of the form $x_1^{j_1} \cdots x_n^{j_n}$ with $j_1 + \cdots +j_n\le (\deg \mu)/d + (q-1)n - (q-1)m/d$. This shows that
\[\deg P \leq  (\deg \mu)/d +(q-1)n - (q-1)m/d = (q-1)(n-m/d) + (\deg \mu)/d,\]
as desired.
\end{proof}

Lastly, we recall the key lemma in the Croot--Lev--Pach \cite{CLP17} polynomial method.

\begin{lemma}\label{lem:CLP}
Let $P \in \F_q[x_1, \ldots, x_n]$ and let $M$ be the $q^n \times q^n$ matrix with rows and columns indexed by elements $\F_q^n$, where for all $(u,v)\in \F_q^n\times\F_q^n$ the $(u,v)$ entry is $M_{uv} = P(u -v)$. Then 
\[ \rank(M) \leq 2 \left| \left \{ (\alpha_1,\ldots, \alpha_n)\in \{0, 1, \ldots, q-1 \}^n:  \alpha_1 + \ldots + \alpha_n \leq \deg P/2 \right \}\right|. \]
\end{lemma}

We provide a proof of this lemma for the sake of completeness. 

\begin{proof}
Let the number of $(\alpha_1,\ldots, \alpha_n) \in \{0, \ldots, q-1 \}^n$ such that $\alpha_1 + \ldots + \alpha_n \leq \deg P/2$ be $T$. First, we claim that it suffices to prove that $M_{uv} = P(u - v) = \sum_{k=1}^{2T} f_k(u) g_k(v)$ for some $f_k, g_k \in \F_q[x_1, \ldots, x_n]$. Indeed, this implies the statement in the lemma because then we can write $M = \sum_{k=1}^{2T} M_k$ where each $M_k$ is a rank 1 matrix. 

Now, let us construct such polynomials $f_k, g_k$ for $k=1,\dots,2T$. Note that each monomial in $P(u-v)$ is of the form $u_1^{a_1} \ldots u_n^{a_n} v_1^{b_1} \ldots v_n^{b_n}$ with $a_1 + \cdots + a_n \leq (\deg P)/2 $ or $b_1 + \cdots + b_n \leq (\deg P)/2$. Consequently, by grouping together monomials with the same factor of degree at most $(\deg P)/2$ in either $u$ or $v$, we may write \[ P(u-v) = \sum_{h}h(u)Q_h(v) + \sum_{h}R_{h}(u)h(v),\]
where the sums are over all monomials $h$ with degree $\deg h \leq (\deg P)/2$ and $Q_h, R_h \in \F_q[x_1, \ldots, x_n]$ are polynomials indexed by the monomials $h$. The number of such monomials $h$ is $T$, so this gives the desired expression of the form $P(u - v) = \sum_{k=1}^{2T} f_k(u) g_k(v)$.
\end{proof}

\section{S\'ark\"ozy's theorem in \texorpdfstring{$\mathbb{F}_q[x]$}{Fq[x]}}

In this section, we prove Theorem~\ref{thm:S}, which strengthens Theorem~\ref{thm:Green-S} due to Green \cite{G17}. In his proof, Green encodes the polynomial $F$ via a map $\Phi \colon \F_q^m \to \F_q^n$ in such a way that the image of $\Phi$ corresponds to the set of polynomials of the form $F(b(x))$ for some $b(x)\in \F_q[x]$ of degree at most $m$, and such that $|\Phi^{-1}(0)|$ is the number of roots of $F$ in $\F_q$. Green's proof proceeds by constructing a polynomial $P\in \F_q[x_1,\dots,x_n]$ of relatively low degree such that $P(x) = |\Phi^{-1}(x)|$ for all $x\in \F_q^n$, and hence $\supp(P) \subset \im(\Phi)$. One can then apply Lemma~\ref{lem:CLP} to the polynomial $P$, obtaining an upper bound for the rank of the $q^n \times q^n$ matrix $M$ indexed by elements of $\F_q^n$, where the $(u,v)$ entry is given by $P(u-v)$. A set $A$ as in Theorem~\ref{thm:Green-S} gives rise to a subset $A\subseteq \F_q^n$ satisfying $(A-A) \cap \mathrm{im}(\Phi) = \{ 0 \}$. For such a subset $A\subseteq \F_q^n$, one can show that the $|A|\times |A|$ submatrix of $M$ indexed by the elements of $A$ is diagonal (i.e. all off-diagonal entries in this submatrix are zero). The diagonal entries of this submatrix are all equal to $P(0)=|\Phi^{-1}(0)|$. Thus, if $|\Phi^{-1}(0)|$ is non-zero in $\F_q$ (i.e.\ if the number of roots of $F$ in $\F_q$ is coprime to $q$), this $|A|\times |A|$ submatrix has full rank $|A|$, and so $|A|$ is at most the rank of the matrix $M$ (which together with Lemma~\ref{lem:CLP} gives the desired bound for $|A|$). However, if $|\Phi^{-1}(0)|$ is zero in $\F_q$, this $|A|\times |A|$ submatrix is all-zero and we cannot make any useful conclusions from the bound for the rank of $M$. This is why the assumption on the number of roots of $F$ in $\F_q$ being coprime to $q$ is needed in Green's proof.

In our proof, we retain the same construction of $\Phi\colon \F_q^m \to \F_q^n$. However, we replace Green's construction of $P$ with a different construction of a polynomial $P$ with $\supp(P) \subset \im(\Phi)$, still of relatively low degree, ensuring that $P(0)$ is non-zero in $\F_q$ without making Green's assumption that the number of roots of $F$ in $\F_q$ is coprime to $q$. The construction of our polynomial $P$ relies on Lemmas~\ref{lem:weight1} and \ref{lem:Poly} proved in the previous section.

\begin{proof}[Proof of Theorem~\ref{thm:S}]
We begin with a similar reduction as in \cite{G17}. Let $P_{q,n}$ denote the set of all polynomials in $\F_q[T]$ with degree less than $n$. Let $m = \lfloor (n-1)/k \rfloor + 1 \geq n/k$. We identify $P_{q,n}$ with $\F_q^n$ by mapping a polynomial to an $n$-tuple representing its coefficients. Specifically, we map $c_0 + c_1x + \cdots + c_{n-1}x^{n-1}$ to $(c_0, c_1, \ldots, c_{n-1})$. Under this identification, we encode $h(x) \mapsto F(h(x))$ as a polynomial map $\Phi: \F_q^{m} \to \F_q^{n}$. More specifically, $\Phi \colon \F_q^m \to \F_q^n$ is given by
\[ \Phi(c_0, \ldots, c_{m-1}) = (\phi_0(c_0, \ldots, c_{m-1}), \ldots, \phi_{n-1}(c_0, \ldots, c_{m-1})), \]
where the polynomials $\phi_i\in \F_q[x_1, \ldots, x_m]$ are specified by $F(c_0 + c_1x + \cdots + c_{m-1}x^{m-1}) = \phi_0(c_0, \ldots, c_{m-1}) + \phi_1(c_0, \ldots, c_{m-1}) x+\cdots + \phi_{n-1}(c_0, \ldots, c_{m-1})x^{n-1}$. Here, we note that by our assumption that $\deg F = k$, we have $\deg(F(c_0 + c_1x + \cdots +c_{m-1}x^{m-1})) \leq k(m-1)=k\cdot \lfloor (n-1)/k\rfloor \leq n-1$. 

We observe some properties of $\Phi$. First, note that $\Phi^{-1}(0)$ corresponds to the roots of $F$ in $\F_q$. Since $F$ has constant term zero, it follows that $0 \in \Phi^{-1}(0)$. Furthermore, because $\deg F \leq k$, it has at most $k$ roots in $\F_q$, and so $|\Phi^{-1}(0)|\leq k$. 

Next, we claim that $\deg \phi_i \leq \min\{ k, (q-1)(1+ \log_q k) \}$ for $0 \leq i \leq n-1$ (which was also observed in \cite{G17}). It is clear that $ \deg \phi_i \leq k$ since $\deg F = k$. To show that $\deg \phi_i \leq (q-1)(1+ \log_qk)$, start by writing
\[(c_0 + c_1x + \ldots + c_{m-1}x^{m-1})^t = \prod_j\left(c_0 + c_1x^{q^j} + \cdots + c_{m-1}x^{(m-1)q^j} \right)^{t_j}\]
where $t = (\ldots t_2t_1t_0)_{q}$ is the base $q$ expansion of $t$. In particular, this shows that $\deg \phi_i \leq \max_{t \leq k} D_q(t) \leq (q-1)(1+ \log_q k)$ where $D_q(t)$ is the sum of the digits of $t$ under base $q$ expansion.

Now, suppose $A \subset P_{q,n}$ is a non-empty subset of polynomials such that there do not exist distinct polynomials $p_1(x),p_2(x) \in A$  with $p_1(x) - p_2(x) = F(b(x))$ for some $b(x) \in \F_q[x]$. Under the identification of $P_{q,n}$ with $\F_q^n$ as explained above, we can instead consider $A$ as a subset of $\F_q^n$ with the property that $(A-A) \cap \im(\Phi) = \{ 0 \}$. We can apply Lemma~\ref{lem:weight1} to the set $\Phi^{-1}(0)$ which gives us a polynomial $\mu\in \F_q[x_1,\dots,x_m]$ with the property $\deg \mu \leq |\Phi^{-1}(0)| - 1 \leq k-1$ such that $\sum_{a \in \Phi^{-1}(0)} \mu(a) \neq 0$. Let $P \in \F_q[x_1, \ldots, x_n]$ be the polynomial obtained from using this polynomial $\mu$ in Lemma~\ref{lem:Poly} and setting $d = \min\{ k, (q-1)(1+ \log_q k) \}$. Recall $P$ has the following properties:
\begin{itemize}
    \item $P(0) \neq 0$,
    \item $\supp(P) \subset \im(\Phi)$, and
    \item $\deg P \leq (q-1)(n- m/d) + (k-1)/d$.
\end{itemize}

Now, as in Lemma~\ref{lem:CLP} applied to the polynomial $P$, consider the $q^n \times q^n $ matrix $M$ indexed by elements of $\F_q^n$, where for all $(u,v)\in \F_q^n \times \F_q^n$ the $(u,v)$ entry is given by $P(u-v)$. Since $P(0) \neq 0$, all diagonal entries of $M$ are nonzero. We claim that in the $\abs{A} \times \abs{A}$ submatrix of $M$ indexed by elements of $A$, all off-diagonal entries are zero. Indeed, for any distinct $u,v \in A$ we have we have $u-v \in (A-A)\setminus \{0\}$. As $(A-A) \cap \im (\Phi) = \{ 0 \}$, it follows that $u-v \not \in \im(\Phi)$ and therefore $u-v \not \in \supp(P)$, which means that $M_{uv} = P(u-v) = 0$. So indeed all off-diagonal entries of the $\abs{A} \times \abs{A}$ submatrix of $M$ indexed by elements of $A$ are zero. Therefore this submtarix is a diagonal matrix and has rank equal to $|A|$. We can conclude that $|A|\le \rank(M)$, and together with Lemma~\ref{lem:CLP} we obtain the bound \[|A| \leq \rank(M) \leq 2\abs{\left\{(\alpha_1,\ldots, \alpha_n)\in \{0, 1, \ldots, q-1 \}^n : \alpha_1 + \cdots + \alpha_n \leq \frac{1}{2} \left( (q-1)\left( n- \frac{m}{d}\right) + \frac{k-1}{d} \right ) \right\}}.\] Note that this last expression is equal to the sum of all the coefficients corresponding to monomials of degree at most $\frac{1}{2} \left( (q-1)(n-m/d) + (k-1)/d \right)$ in the expansion of $(1+x+ \cdots + x^{q-1})^n$. In particular, for every $0 < x <1$, we have
\[ \frac{|A|}{2} \cdot  x^{\frac{1}{2} \left( (q-1)(n-m/d) + (k-1)/d \right)} \leq (1+x+ \cdots +x^{q-1})^n.\]
We can conclude that
\[\abs{A} \leq 2\cdot  \frac{(1+x+ \cdots + x^{q-1})^n}{x^{\frac{1}{2} \left( (q-1)(n-m/d) + (k-1)/d \right)}}\le 2 \cdot \frac{(1+x+ \cdots + x^{q-1})^n}{x^{\frac{1}{2}(q-1)(n-n/(kd)) + (k-1)/(2d)}}= \frac{2}{x^{(k-1)/(2d)}} \cdot \left(\frac{1+x+ \cdots + x^{q-1}}{x^{\frac{1}{2}(q-1)(1-1/(kd)) }}\right)^n\]
for every $0 < x <1$. Note that the infimum in (\ref{eq:t-exp}) is actually attained by some value $0<x<1$ (indeed, the expression on the right side goes to infinity for $x\to 0$, and the derivative of this expression is positive at $x=1$). For this value of $x$ (which depends only on $q$ and $k$, but not on $n$), we obtain
\[ \abs{A} \leq \frac{2}{x^{(k-1)/(2d)}} \cdot \left(\frac{1+x+ \cdots + x^{q-1}}{x^{\frac{1}{2}(q-1)(1-1/(kd)) }}\right)^n = c_{q,k} \cdot (t_{q,k})^{n},\]
where $c_{q,k}=2/x^{(k-1)/(2d)}$ is a constant only depending on $q$ and $k$.
\end{proof}

\section{S\'ark\"ozy's Theorem in \texorpdfstring{$\F_{q}$}{Fq}}

In this section, we give an application of Theorem~\ref{thm:S} by proving Corollary~\ref{cor:Fq}, which is a variant of S\'ark\"ozy's Theorem in $\F_q$, where $q = p^n$ is a prime power.

\begin{proof}[Proof of Corollary~\ref{cor:Fq}]
Suppose $A \subset \F_q$ is a non-empty subset that does not contain distinct elements $a_1,a_2 \in A$ such that $a_1 - a_2 = F(b)$ for some $b \in \F_q$.

We make an identification of $\F_q$ with the polynomial ring setting in Theorem~\ref{thm:S}. Note that $\F_q$ can be written as a $n$-dimensional vector space over $\F_p$ with basis $1, \beta, \ldots, \beta^{n-1}$ where $\beta$ is the root of any degree $n$ irreducible polynomial over $\F_p$. In particular, this means that we can think of elements of $\F_p$ as polynomials of degree less than $n$ in $\F_p[\beta]$. 

Applying Theorem~\ref{thm:S} to $A$ under this identification of $\F_q$ with $\F_p[\beta]$, we get $|A| \leq c_{p,k}\cdot (t_{p,k})^n$ as desired. \end{proof}

\bibliographystyle{plain}
\bibliography{ref.bib}

\end{document}